\documentclass[11pt,reqno]{amsart}
\usepackage[T1]{fontenc}
\usepackage{lmodern}
\usepackage[letterpaper,margin=1.1in]{geometry}
\usepackage{amsmath,amssymb,mathtools}
\usepackage[expansion=false]{microtype}
\usepackage[hidelinks]{hyperref}

\newtheorem{theorem}{Theorem}[section]
\newtheorem{proposition}[theorem]{Proposition}
\newtheorem{lemma}[theorem]{Lemma}
\newtheorem{corollary}[theorem]{Corollary}
\theoremstyle{definition}
\newtheorem{definition}[theorem]{Definition}
\theoremstyle{remark}
\newtheorem{remark}[theorem]{Remark}
\numberwithin{equation}{section}
\DeclareMathOperator{\SL}{SL}
\DeclareMathOperator{\GL}{GL}
\DeclareMathOperator{\Sym}{Sym}

\DeclareMathOperator{\id}{id}
\newcommand{\Z}{\mathbb{Z}}

\newcommand{\cB}{\mathcal{B}}

\title[Instability in finite actions]{Infranormal Kazhdan subgroups\protect\\
and instability in finite actions}
\author{Jonas Pinke}
\date{}

\begin{document}

\begin{abstract}
We prove that the group
\(G=\Z[1/2]^3\rtimes(\SL_3(\Z)\times\Z)\)
is finitely presented and residually finite, but is not stable in
finite actions, answering a question of Gohla and Thom. The proof
uses the infranormal Kazhdan subgroup
\(H=\Z^3\rtimes\SL_3(\Z)\) and its normal closure \(N\) in \(G\).
A Kazhdan projection estimate shows that, in any \(G\)-action
weakly contained in finite actions, every \(H\)-invariant measurable
subset is also \(N\)-invariant.
We construct an explicit sofic approximation of \(G\) whose limit
action contains a measurable subset that is \(H\)-invariant but not
\(N\)-invariant.
\end{abstract}

\maketitle

\section{Introduction}

A residually finite group admits sofic approximations coming from
genuine finite actions. Gohla and Thom call a group \emph{stable in
finite actions} if every sofic limit action is weakly contained in
finite actions; see \cite[Definition~3.7]{GT}. Equivalently, the statistics of finite measurable partitions
in every such limit action must be approximable by those
of genuine finite actions, as in Kechris's notion of weak
containment \cite[Section~10]{Kechris}. They asked whether there exists
a residually finite group that is not stable in finite actions \cite[Question~3.12]{GT}.
We answer this question affirmatively with a finitely presented
example.

Consider the group
\(
 G=\Z[1/2]^3\rtimes\bigl(\SL_3(\Z)\times\Z\bigr),
\)
where \((B,k)\) acts on \(\Z[1/2]^3\) by \(v\mapsto 2^kBv\).
This semidirect product can be realized as an ascending HNN extension of
the Kazhdan group \(H=\Z^3\rtimes\SL_3(\Z)\), with stable letter \(t\)
satisfying \(t(u,B)t^{-1}=(2u,B)\).
The group \(G\) is a particular case of the ascending HNN
extensions considered by Eckhardt in his work on non-finite
full group \(C^*\)-algebras
\cite[Theorem~3.1 and Example~3.2]{Eck}.

Our main theorem is the following.

\begin{theorem}\label{thm:main}
The group \(G\) is finitely presented and residually finite, but is not
stable in finite actions.
\end{theorem}

The proof combines an obstruction for finite actions with an explicit
sofic approximation. The subgroup \(H\) is infranormal in \(G\) in the
sense of Kun and Thom. By their observation in
\cite{KT}, \(H\) and its normal closure
\(N\) have the same image in every finite quotient of \(G\). Hence, in
every finite \(G\)-action, \(H\)-invariant subsets are also
\(N\)-invariant. A Kazhdan estimate controls the \(L^2\)-distance
of the indicator function of a measurable subset from the
\(H\)-invariant functions in terms of its displacements under a fixed
finite generating set of \(H\). Since these functions are also
\(N\)-invariant, the same estimate controls the displacement of the
subset under every element of \(N\). This uniform estimate passes to
metric ultraproducts. Consequently, every \(H\)-invariant measurable
subset in an action weakly contained in finite actions is also
\(N\)-invariant.

We first construct finite models with \(n+1\) equally sized
layers, indexed by \(i=0,\ldots,n\). On each layer, the coordinates are \(\ell\in(\Z/2^n\Z)^{1\times3}\) and
\(x\in\Z/2^n\Z\). There is a genuine \(H\)-action in which
translation by \(u\in\Z^3\) adds \(2^{n-i}\ell u\) to \(x\),
while the linear part \(B\in\SL_3(\Z)\) acts by
\(\ell\mapsto\ell B^{-1}\). All arithmetic is modulo \(2^n\).
The stable letter acts as a cyclic shift between the layers.
The HNN relations hold outside one layer, whose relative size
tends to zero. Taking the product of these models with suitable finite
quotient actions gives a sofic approximation of \(G\).

We also construct subsets \(E_n\) preserved by the \(H\)-actions.
Let \(a\in H\) be translation by the first standard basis vector
\(e_1\), and put \(w=t^{-1}at\), which represents translation by
\(e_1/2\). The permutation representing \(w\) changes membership in
\(E_n\) on a proportion of points tending to \(1/2\). Passing to
a limit action of this sofic approximation, we obtain an
\(H\)-invariant measurable subset \(E\) with
\[
 \mu(wE\triangle E)=\frac12.
\]
Since \(w\in N\), \(E\) is not \(N\)-invariant. The preceding
obstruction therefore shows that this limit action is not weakly
contained in finite actions. Consequently, \(G\) is not stable in finite actions.

\medskip
\noindent\textbf{AI Declaration.}
AI tools were used in preparing this manuscript. In particular, OpenAI's GPT Astra proposed the explicit
\(\SL_3(\Z)\)-equivariant realization of the layered
construction in Section~5 and the subsets \(E_n\)
used in the obstruction.
The layered construction extends the approximation
described in \cite[Remark~4.4]{ThomConditional};
see also Remark~4.3 there. The
author supplied this context and the connection with infranormal
subgroups, and takes responsibility for the mathematical claims,
proofs, and any errors or omissions.

\section{Measured Boolean algebras and stability}\label{sec:stability}

We recall the setup of measured Boolean algebras and the definitions of weak containment and stability in finite actions. All groups are discrete and countable.

\subsection{Measured Boolean algebras and weak containment}

We use the Boolean-algebra formulation of
\cite[Section~2.1]{GT}. A \emph{measured Boolean algebra} is a Boolean
algebra \(\cB\) equipped with a normalized finitely additive probability measure
\(\mu\).
If necessary, we pass to the quotient by null sets, so that
\[
d(B,D)=\mu(B\triangle D)    
\]
defines a metric on \(\cB\).

A measured \(G\)-Boolean algebra is a measured Boolean algebra
together with a \(G\)-action by measure-preserving Boolean
automorphisms. For finite measured \(G\)-Boolean algebras, we simply
speak of \emph{finite actions}. Given a subgroup \(L\leq G\), we write
\(\cB^L\) for the subalgebra corresponding to the \(L\)-invariant
measurable subsets:
\[
 \cB^L=\{B\in\cB:hB=B\text{ for all }h\in L\}.
\]

Given a family of measured Boolean algebras \((\cB_n,\mu_n)_{n\in I}\),
and a non-principal ultrafilter \(\omega\) on \(I\), their metric ultraproduct is
\[
 \prod_{n\to\omega}(\cB_n,\mu_n)
 =\left(\prod_{n\in I}\cB_n\right)\big/\sim,
 \]
 where
 \[
 (B_n)\sim(D_n)
 \ \Longleftrightarrow\
 \lim_{n\to\omega}\mu_n(B_n\triangle D_n)=0.
\]
The equivalence class of \( (B_n)\) is denoted by \([B_n]\). The operations are defined coordinatewise. The measure is given by
\[
\mu_{\omega}([B_n])=\lim_{n \to \omega} \mu_n(B_n).
\]
If each \((\cB_n,\mu_n)\) carries a measure-preserving \(G\)-action, then \(g[B_n]=[gB_n]\) for \(g\in G\) defines a measure-preserving \(G\)-action on the ultraproduct.

Following Gohla and Thom \cite[Definitions~2.1--2.2]{GT}, we make the following definition.

\begin{definition}\label{def:weak}
A measured \(G\)-Boolean algebra is \emph{weakly contained in
finite actions} if it admits a measure-preserving, \(G\)-equivariant
Boolean homomorphism into a metric ultraproduct of finite measured
\(G\)-Boolean algebras.
\end{definition}

Such a homomorphism preserves the metric and is therefore injective.

\subsection{Sofic approximations and their limits}

For a nonempty finite set \(V\), let \(\mu_V\) denote
normalized counting measure. The \emph{normalized Hamming
metric} on \(\Sym(V)\) is bi-invariant and is given by
\[
 d_{\mathrm H}(\tau,\tau')
 =
 \frac{|\{x\in V:\tau x\ne\tau'x\}|}{|V|}
 \qquad(\tau,\tau'\in\Sym(V)).
\]

We use the free-group formulation of sofic approximations;
see \cite[Definition~3.1]{GT}.

\begin{definition}\label{def:sofic}
Let \(q:F(S)\to G\) be the quotient map associated with a
finite presentation \(G=\langle S\mid R\rangle\).
A \emph{sofic approximation} of \(G\) is a sequence of
homomorphisms
\[
 \sigma_n:F(S)\longrightarrow\Sym(V_n),
 \qquad n\geq1,
\]
where each \(V_n\) is a nonempty finite set, satisfying the
following conditions as \(n\to\infty\):
\begin{enumerate}
\item[(i)] \emph{Asymptotic homomorphism.}
For every \(v\in\ker q\),
\[
 d_{\mathrm H}\bigl(\sigma_n(v),\id\bigr)
 \longrightarrow0.
\]

\item[(ii)] \emph{Asymptotic freeness.}
For every \(v\in F(S)\setminus\ker q\),
\begin{equation}\label{eq:asymptotic-free}
 d_{\mathrm H}\bigl(\sigma_n(v),\id\bigr)
 \longrightarrow1.
\end{equation}
\end{enumerate}
\end{definition}

If \(u,v\in F(S)\) represent the same element of \(G\),
then \(uv^{-1}\in\ker q\), and hence
\[
 d_{\mathrm H}\bigl(\sigma_n(u),\sigma_n(v)\bigr)
 =
 d_{\mathrm H}\bigl(\sigma_n(uv^{-1}),\id\bigr)
 \longrightarrow0.
\]
It follows that, for every nonprincipal ultrafilter \(\omega\) on \( \mathbb{N}\) and
every sequence of subsets \(B_n\subseteq V_n\), the limit
\[
 \lim_{n\to\omega}
 \mu_{V_n}\bigl(\sigma_n(v)B_n\triangle B_n\bigr)
\]
is independent of the word \(v\)
representing a given element of \(G\).

\begin{lemma}\label{lem:limit-action}
Every sofic approximation and every nonprincipal
ultrafilter \(\omega\) on \(\mathbb N\) induce a
measure-preserving action of \(G\) on
\[
 (\cB_\omega,\mu_\omega)
 =
 \prod_{n\to\omega}
 \bigl(\mathcal P(V_n),\mu_{V_n}\bigr).
\]
\end{lemma}

\begin{proof}
The homomorphisms \(\sigma_n\) define measure-preserving
\(F(S)\)-actions on \(\mathcal P(V_n)\), and therefore
on their metric ultraproduct. For every \(v\in\ker q\)
and every element \([B_n]\), we have
\[
 \mu_\omega\bigl(v[B_n]\triangle[B_n]\bigr)=\lim_{n\to\omega}\mu_{V_n}\bigl(\sigma_n(v)B_n\triangle B_n\bigr)\leq\lim_{n\to\omega}d_{\mathrm H}\bigl(\sigma_n(v),\id\bigr)=0.
\]
Thus \(\ker q\) acts trivially, so the action factors
through \(G\).
\end{proof}

We recall Gohla and Thom's notion of stability in finite
actions \cite[Definition~3.7]{GT}.

\begin{definition}\label{def:stable}
A finitely presented group \(G\) is \emph{stable in finite
actions} if, for every sofic approximation of \(G\) and
every nonprincipal ultrafilter \(\omega\) on \(\mathbb N\),
the induced limit action is weakly contained in finite
actions.
\end{definition}

For measure algebras of probability-preserving actions,
the notion of weak containment used here corresponds
to that of Kechris \cite[Section~10]{Kechris};
see also the discussion following
\cite[Definition~2.2]{GT}.
An equivalent formulation of stability in finite actions,
using finite partitions of sufficiently accurate sofic
approximations and their approximation in genuine finite
actions, is given in
\cite[Lemma~3.8, \((i)\Leftrightarrow(ii)\)]{GT}.

\section{An infranormal finite-action obstruction}\label{sec:obstruction}

\subsection{Images in finite quotients}

The following definition is due to Kun and Thom~\cite{KT}.

\begin{definition}
Let \( G\) be a group and let \(H\leq G\) be a subgroup. The semigroup
\[
 P_H(G)=\{g\in  G:gHg^{-1}\leq H\}
\]
is called the \emph{compression semigroup} of \(H\) in \(G\).
We call \(H\) an \emph{infranormal} subgroup of \( G\) if
\(P_H(G)\) generates \(G\) as a group.
\end{definition}

The following observation is due to Kun and Thom; see the remark
following Theorem~C in \cite{KT}.

\begin{lemma}\label{lem:finite-quotient}
Let \(H\) be an infranormal subgroup of \(G\), and let
\(N=\langle\!\langle H\rangle\!\rangle\) be the normal closure
of \(H\) in \(G\). For every homomorphism \(q:G\to F\) to a
finite group \(F\), the subgroup \(q(H)\) is normal in \(q(G)\),
and \(q(H)=q(N)\).
\end{lemma}

\begin{proof}
Let \(g\in P_H(G)\). Since \(gHg^{-1}\leq H\), we have
\(
 q(g)q(H)q(g)^{-1}\leq q(H).
\)
These finite subgroups have the same order, so the inclusion is
an equality. Thus \(q(g)\) normalizes \(q(H)\) for every
\(g\in P_H(G)\). Since \(P_H(G)\) generates \(G\) as a group,
its image generates \(q(G)\). Hence \(q(H)\) is normal in \(q(G)\).
Finally, \(q(N)\) is the normal closure of \(q(H)\) in \(q(G)\) and hence \(q(H)=q(N)\).
\end{proof}

Consequently, for every action of \(G\) on a finite set \(Y\),
every \(H\)-invariant subset of \(Y\) is also \(N\)-invariant.
This follows by applying Lemma~\ref{lem:finite-quotient} to the
homomorphism \(G\to\Sym(Y)\) defining the action.

\subsection{A Kazhdan projection estimate}

Recall that a discrete group \(H\) has property~\((T)\) if there
exist a finite subset \(Q\subseteq H\) and \(\varepsilon>0\)
such that every unitary representation \((\pi,\mathcal{H})\) of \(H\) admitting a
nonzero vector \(\xi \in \mathcal{H}\) with
\[
 \max_{h\in Q}\|\pi(h)\xi-\xi\|<\varepsilon\|\xi\|
\]
has a nonzero invariant vector. The pair \((Q,\varepsilon)\)
is called a \emph{Kazhdan pair}; see
\cite[Definition~1.1.3]{BHV}.

For the remainder of this section, let \(H\) be finitely generated with property~\((T)\), and fix a
finite symmetric generating set \(S_H\). By
\cite[Proposition~1.3.2(i)]{BHV}, there exists \(\varepsilon>0\)
such that \((S_H,\varepsilon)\) is a Kazhdan pair. Fix such an \(\varepsilon\) and set
\(\kappa=\varepsilon^2\). For any unitary representation \((\pi,\mathcal{H})\)
of \(H\), let \(P\) denote the
orthogonal projection onto the subspace \(\mathcal{H}^{H}\) of \(H\)-invariant vectors.

The orthogonal complement \((\mathcal H^H)^{\perp}\) is \(H\)-invariant
and contains no nonzero invariant vectors. By the definition of
a Kazhdan pair, every \(\eta \in (\mathcal{H}^{H})^{\perp} \) satisfies
\(\max_{s\in S_H}\|\pi(s)\eta-\eta\|\geq\varepsilon\|\eta\|\).
Taking \(\eta=\xi-P\xi\) and using \(\pi(s)P\xi=P\xi\), we obtain
\begin{equation}\label{eq:kazhdan}
 \kappa\|\xi-P\xi\|^2
 \leq\max_{s\in S_H}\|\pi(s)\xi-\xi\|^2
 \qquad(\xi\in\mathcal H).
\end{equation}
See also \cite[Proposition~1.1.9]{BHV}. A related application of this estimate to almost invariant
finite partitions appears in \cite[Lemma~3.3]{AT}.

\begin{proposition}\label{prop:finite-inequality}
Let \(G\) be a group containing \(H\) as an infranormal subgroup, and let
\(N=\langle\!\langle H\rangle\!\rangle\) be the normal
closure of \(H\) in \(G\).
For every finite measure-preserving \(G\)-action
\((Y,\nu)\), every \(B\subseteq Y\), and every \(g\in N\),
we have
\begin{equation}\label{eq:finite-inequality}
 \nu(gB\triangle B)
 \leq \frac{4}{\kappa} \max_{s\in S_H}\nu(sB\triangle B).
\end{equation}
\end{proposition}

\begin{proof}
Let \(\pi\) be the Koopman representation of \(G\) on
\(L^2(Y,\nu)\), defined by
\(
 (\pi(g)f)(y)=f(g^{-1}y).
\)
Let \(P\) be the orthogonal projection onto the \(H\)-invariant
functions, and put \(f=\mathbf1_B\).
Let \(q:G\to\Sym(Y)\) be the homomorphism defining the action.
By Lemma~\ref{lem:finite-quotient}, \(q(H)=q(N)\).
Hence every \(H\)-invariant function is also \(N\)-invariant,
so \(\pi(g)Pf=Pf\) for every \(g\in N\).
Unitarity and \eqref{eq:kazhdan}, applied to the restriction of
\(\pi\) to \(H\), now give
\begin{align*}
 \nu(gB\triangle B)
 &=\|\pi(g)f-f\|_2^2=\|\pi(g)(f-Pf)-(f-Pf)\|_2^2\\
 &\leq4\|f-Pf\|_2^2\leq\frac4\kappa\max_{s\in S_H}\|\pi(s)f-f\|_2^2=\frac4\kappa\max_{s\in S_H}\nu(sB\triangle B).\qedhere
\end{align*}
\end{proof}

\subsection{Passage to weak containment}

\begin{corollary}\label{cor:weak-inequality}
Let \(G\) be a group containing \(H\) as an infranormal
subgroup, and let \(N=\langle\!\langle H\rangle\!\rangle\)
be the normal closure of \(H\) in \(G\).
Let \((\cB,\mu)\) be a measured \(G\)-Boolean algebra
weakly contained in finite actions.
Then, for every \(B\in\cB\) and every \(g\in N\), we have
\begin{equation}\label{eq:weak-inequality}
 \mu(gB\triangle B)
 \leq\frac{4}{\kappa}
 \max_{s\in S_H}\mu(sB\triangle B).
\end{equation}
In particular, \(\cB^H=\cB^N\).
\end{corollary}

\begin{proof}
By weak containment, choose a measure-preserving
\(G\)-equivariant embedding of \((\cB,\mu)\) into a
metric ultraproduct
\(
 \prod_{n\to\omega}\bigl(\mathcal P(Y_n),\nu_n\bigr)
\)
associated to finite measure-preserving \(G\)-actions
\((Y_n,\nu_n)\).
Fix \(B\in\cB\), and represent its image by \([B_n]\).
For each \(g\in N\), apply
Proposition~\ref{prop:finite-inequality} to \(B_n\) and
take the ultralimit along \(\omega\). Since \(S_H\) is
finite, the maximum commutes with the ultralimit, giving
\eqref{eq:weak-inequality}.
If \(B\in\cB^H\), the right-hand side vanishes, so
\(B\in\cB^N\). The reverse inclusion follows from
\(H\leq N\).
\end{proof}
For an infranormal subgroup \(H\leq G\), Kun and Thom prove
that the algebra of \(H\)-invariant functions is \(G\)-invariant
in every sofic probability-measure-preserving action when
both groups have property~\((T)\) \cite[Theorem~C]{KT}.
Corollary~\ref{cor:weak-inequality} concerns actions weakly
contained in finite actions and requires property~\((T)\)
only for \(H\).

In the following corollary, \(\mu_n\) denotes normalized
counting measure on \(V_n\), and \(\sigma_n(g)\) denotes
the evaluation under \(\sigma_n\) of a fixed word
representing \(g\in G\). The asymptotic conditions are
independent of these choices.

\begin{corollary}\label{cor:criterion}
Let \(G\) be a finitely presented group containing \(H\)
as an infranormal subgroup, and let
\(N=\langle\!\langle H\rangle\!\rangle\) be the normal
closure of \(H\) in \(G\).
Suppose that a sofic approximation \((\sigma_n)\) of \(G\)
on finite sets \(V_n\) admits subsets \(E_n\subseteq V_n\)
and an element \(w\in N\) such that
\[
 \begin{aligned}
 \mu_n\bigl(\sigma_n(s)E_n\triangle E_n\bigr)
 &\longrightarrow 0
 &&\text{for every }s\in S_H,\\
 \mu_n\bigl(\sigma_n(w)E_n\triangle E_n\bigr)
 &\longrightarrow c>0.
 \end{aligned}
\]
Then, for every nonprincipal ultrafilter \(\omega\) on
\(\mathbb N\), the induced limit action is not weakly
contained in finite actions. In particular, \(G\) is not
stable in finite actions.
\end{corollary}

\begin{proof}
Fix a nonprincipal ultrafilter \(\omega\) on \(\mathbb N\),
and let \((\cB_\omega,\mu_\omega)\) be the induced limit
action. Set \(E=[E_n]\). The first assumption gives
\(
 \mu_\omega(sE\triangle E)=0
\)
for every \(s\in S_H\).
Since \(S_H\) generates \(H\), we have
\(E\in\cB_\omega^H\). The second assumption gives
\(
 \mu_\omega(wE\triangle E)=c>0.
\)
If this limit action were weakly contained in finite
actions, Corollary~\ref{cor:weak-inequality} would imply
\(E\in\cB_\omega^N\), contradicting the last equality
because \(w\in N\). The final assertion follows from
Definition~\ref{def:stable}.
\end{proof}

\section{The affine ascending HNN extension}\label{sec:group}

Let \(H=\Z^3\rtimes\SL_3(\Z)\), where the multiplication is given by
\(
 (u,B)(v,D):=(u+Bv,BD).
\)
The group \(H\) is finitely presented and has Kazhdan's
property~\((T)\); see \cite[Proposition~1.1]{Cornulier}.

The map \(\varphi:H\to H\), defined by \(\varphi(u,B)=(2u,B)\), is an
injective endomorphism whose image has index~\(8\). We define the associated ascending HNN extension by
\begin{equation}\label{eq:hnn}
 G_0
 =\langle H,t\mid tht^{-1}=\varphi(h)\ (h\in H)\rangle.
\end{equation}
For background on HNN extensions, see Lyndon--Schupp~\cite[Chapter~IV, \S\,2]{LS}.

\begin{proposition}\label{prop:hnn}
The group \(G_0\) is isomorphic to
\(G=\Z[1/2]^3\rtimes\bigl(\SL_3(\Z)\times\Z\bigr)\).
The isomorphism \(\Phi:G_0\to G\) is determined by
\(\Phi((u,B))=(u,B,0)\) and \(\Phi(t)=(0,I_3,1)\).
Every element of \(G_0\) has an expression \(t^{-r}(u,B)t^s\)
with \((u,B)\in H\) and \(r,s\geq0\), and
\begin{equation}\label{eq:collected-form}
 \Phi\bigl(t^{-r}(u,B)t^s\bigr)=(2^{-r}u,B,s-r).
\end{equation}
\end{proposition}

\begin{proof}
Multiplication in \(G\) is given by
\begin{equation}\label{eq:group-law}
 (v,B,k)(v',B',k')
 =(v+2^kBv',BB',k+k').
\end{equation}
The prescribed assignments respect multiplication in \(H\) and
the HNN relations, so they define a homomorphism \(\Phi:G_0\to G\).
The relations \(th=\varphi(h)t\) and
\(ht^{-1}=t^{-1}\varphi(h)\) move \(t\) to the right and \(t^{-1}\)
to the left across elements of \(H\). Thus every element of \(G_0\)
has an expression \(t^{-r}(u,B)t^s\) with \(r,s\geq0\), whose image
under \(\Phi\) is given by \eqref{eq:collected-form}.
A trivial image forces \(u=0\), \(B=I_3\), and \(s=r\), so the
original element is trivial. Thus \(\Phi\) is injective.
For surjectivity, given \((v,B,k)\in G\), choose \(r\geq0\) so that
\(2^rv\in\Z^3\) and \(r+k\geq0\). Then
\(\Phi\bigl(t^{-r}(2^rv,B)t^{r+k}\bigr)=(v,B,k)\).
\end{proof}

We identify \(G_0\) with \(G\) via \(\Phi\).
A finite presentation of \(G\) is obtained from a finite presentation
of \(H\) by adding \(t\) and the relations
\(tst^{-1}=\varphi(s)\) for the generators \(s\) of that presentation.
The map
\begin{equation}\label{eq:affine-embedding}
 G\longrightarrow\GL_4(\Z[1/2]),
 \qquad
 (v,B,k)\longmapsto
 \begin{pmatrix}
  2^kB&v\\
  0&1
 \end{pmatrix}
\end{equation}
is a homomorphism by \eqref{eq:group-law}.
It is injective because the determinant \(2^{3k}\) determines \(k\),
and the remaining entries determine \(B\) and \(v\).
Reduction modulo odd primes separates points, so \(G\), and hence
\(H\), is residually finite.

\begin{lemma}\label{lem:group-structure}
The subgroup \(H\) is infranormal in \(G\) but is not normal.
We have
\begin{equation}\label{eq:strict-inclusions}
 tHt^{-1}
 =2\Z^3\rtimes\SL_3(\Z)
 \subsetneq H,
 \qquad
 t^{-1}Ht
 =\tfrac12\Z^3\rtimes\SL_3(\Z)
 \supsetneq H.
\end{equation}
Let \(\chi:G\to\Z\) be the homomorphism given by
\(\chi(v,B,k)=k\).
Then the normal closure \(N=\langle\!\langle H\rangle\!\rangle\)
of \(H\) in \(G\) satisfies
\begin{equation}\label{eq:normal-closure}
 N
 =\bigcup_{r\geq0}t^{-r}Ht^r
 =\Z[1/2]^3\rtimes\SL_3(\Z)
 =\ker\chi.
\end{equation}
Let \(e_1\) be the first standard basis vector. If we set \(a=(e_1,I_3)\in H\) and \(w=t^{-1}at\),
then \(w\) corresponds to translation by \(\tfrac12e_1\)
and belongs to \(N\setminus H\).
\end{lemma}

\begin{proof}
The two conjugation identities follow from
\eqref{eq:collected-form}.
Every element of \(H\) normalizes \(H\), so \(H \subseteq P_H(G)\). Moreover, the inclusion \(tHt^{-1}\leq H\) gives \(t\in P_H(G)\). Since \(G=\langle H,t \rangle\), the compression semigroup \(P_H(G)\) generates \(G\) as a group. Thus \(H\) is infranormal in \(G\).
The strict inclusions in \eqref{eq:strict-inclusions}
show that \(H\) is not normal.

For every \(r\geq0\), equation~\eqref{eq:collected-form} gives
\[
 t^{-r}Ht^r
 =2^{-r}\Z^3\rtimes\SL_3(\Z).
\]
These subgroups form an increasing sequence whose union is
\(\Z[1/2]^3\rtimes\SL_3(\Z)=\ker\chi\).
Since \(\ker\chi\) is normal in \(G\) and contains \(H\),
we have \(N\leq\ker\chi\).
Conversely, every conjugate \(t^{-r}Ht^r\) is contained in
\(N\), so \(\ker\chi\leq N\).
This proves \eqref{eq:normal-closure}.

Finally, \eqref{eq:collected-form} gives
\(w=(\tfrac12e_1,I_3,0)\), so \(w\in N\setminus H\).
\end{proof}

\section{The layered approximation and an invariant measurable subset}
\label{sec:models}

The construction below extends the layered approximation
of \(BS(1,2)\) described in
\cite[Remark~4.4]{ThomConditional} by incorporating
the action of \(\SL_3(\Z)\).

\subsection{Finite permutations}

For \(n\geq1\), we set
\begin{equation}\label{eq:model-sets}
 R_n=\Z/2^n\Z,
 \qquad
 X_n=\{0,1,\ldots,n\}\times R_n^{1\times3}\times R_n.
\end{equation}
We write the points of \(X_n\) as \((i,\ell,x)\).
Each layer \(\{i\}\times R_n^{1\times3}\times R_n\)
has cardinality \(2^{4n}\).
Bars denote reduction modulo \(2^n\).
For \((u,B)\in H\), define
\begin{equation}\label{eq:rho}
 \rho_n(u,B)(i,\ell,x)
 =\bigl(i,\ell\overline B^{-1},
 x+2^{n-i}\ell\overline B^{-1}\overline u\bigr),
\end{equation}
and put
\begin{equation}\label{eq:shift}
 T_n(i,\ell,x)
 =\bigl(i+1\bmod(n+1),\ell,x\bigr).
\end{equation}
Arithmetic in the last two coordinates is performed over the ring \(R_n\).
In particular, the translation term in \eqref{eq:rho}
vanishes on layer~\(0\).
The shift \(T_n\) is cyclic and leaves the coordinates
\(\ell\) and \(x\) unchanged.

\begin{lemma}\label{lem:rho}
The map \(\rho_n:H\to\Sym(X_n)\) is a homomorphism.
For every \(h\in H\),
\begin{equation}\label{eq:hnn-error}
 d_{\mathrm H}\bigl(T_n\rho_n(h)T_n^{-1},
                   \rho_n(\varphi(h))\bigr)
 \leq\frac1{n+1}.
\end{equation}
\end{lemma}

\begin{proof}
For \((u,B),(v,D)\in H\), composing the maps in \eqref{eq:rho} gives
\[
 \rho_n(u,B)\rho_n(v,D)(i,\ell,x)
 =\bigl(i,\ell\overline{BD}^{-1},
 x+2^{n-i}\ell\overline{BD}^{-1}\overline{u+Bv}\bigr)
 =\rho_n(u+Bv,BD)(i,\ell,x).
\]
Taking \((v,D)=(u,B)^{-1}\) shows that
\(\rho_n((u,B)^{-1})\) is the inverse of \(\rho_n(u,B)\).
Hence the maps are permutations and
\(\rho_n:H\to\Sym(X_n)\) is a homomorphism.

Write \(h=(u,B)\in H\), so that \(\varphi(h)=(2u,B)\).
For \(i<n\), the permutations \(T_n\rho_n(h)\) and
\(\rho_n(\varphi(h))T_n\) both send \((i,\ell,x)\) to
\[
 \bigl(i+1,\ell\overline B^{-1},
 x+2^{n-i}\ell\overline B^{-1}\overline u\bigr),
\]
because \(2^{n-i-1}(2\overline u)=2^{n-i}\overline u\).
For \(i=n\), however,
\begin{align*}
 T_n\rho_n(h)(n,\ell,x)
 &=\bigl(0,\ell\overline B^{-1},
 x+\ell\overline B^{-1}\overline u\bigr),\\
 \rho_n(\varphi(h))T_n(n,\ell,x)
 &=\bigl(0,\ell\overline B^{-1},x\bigr).
\end{align*}
Thus the two permutations agree away from one layer.
Bi-invariance of the Hamming metric proves \eqref{eq:hnn-error}.
\end{proof}

To ensure asymptotic freeness, we take the product with
actions on finite quotients of \(G\), as in
\cite[Lemma~3.2]{AT}.
Since \(G\) is residually finite, we may choose finite quotient maps
\(p_n:G\to Q_n\) such that, for every \(g\neq1\), we have
\(p_n(g)\neq1\) for all sufficiently large \(n\).
Set
\begin{equation}\label{eq:product-model}
 V_n=X_n\times Q_n
\end{equation}
and equip \(V_n\) with normalized counting measure \(\mu_n\).
For \(h\in H\), \(z\in X_n\), and \(y\in Q_n\), define
\begin{equation}\label{eq:product-permutations}
 \begin{aligned}
 \widetilde{\rho}_n(h)(z,y)
 &=\bigl(\rho_n(h)z,p_n(h)y\bigr),\\
 \widetilde{T}_n(z,y)
 &=\bigl(T_nz,p_n(t)y\bigr).
 \end{aligned}
\end{equation}
The map \(\widetilde{\rho}_n:H\to\Sym(V_n)\) is a homomorphism.
Since \(p_n\) is a homomorphism, the relation \(tht^{-1}=\varphi(h)\) holds exactly on the \(Q_n\)-coordinate. Therefore,
\begin{equation}\label{eq:product-hnn-error}
 d_{\mathrm H}\bigl(\widetilde{T}_n\widetilde{\rho}_n(h)\widetilde{T}_n^{-1},\widetilde{\rho}_n(\varphi(h))\bigr)\leq\frac1{n+1}.
\end{equation}

Fix a finite presentation of \(H\) with finite generating set \(S_H\),
and let \(S=S_H\cup\{t\}\) be the generating set of the associated
HNN presentation of \(G\) described in Section~\ref{sec:group}.
Let \(q:F(S)\to G\) denote the associated quotient map.
Define the homomorphism
\[
 \sigma_n:F(S)\longrightarrow\Sym(V_n)
\]
on the free generators by
\(\sigma_n(s)=\widetilde{\rho}_n(s)\) for \(s\in S_H\)
and \(\sigma_n(t)=\widetilde{T}_n\).

\begin{proposition}\label{prop:sofic}
The sequence \((\sigma_n)_{n\geq1}\) is a sofic approximation of \(G\).
\end{proposition}

\begin{proof}
The asymptotic homomorphism condition follows from
\eqref{eq:product-hnn-error} and bi-invariance of the Hamming metric.
To verify asymptotic freeness, fix \(g\neq1\)
and \(v\in F(S)\) with \(q(v)=g\).
The permutation \(\sigma_n(v)\) acts on the \(Q_n\)-coordinate
by \(y\mapsto p_n(g)y\).
For all sufficiently large \(n\), we have \(p_n(g)\neq1\),
so \(\sigma_n(v)\) has no fixed points and
\(
 d_{\mathrm H}\bigl(\sigma_n(v),\mathrm{id}\bigr)=1.
\)
This proves asymptotic freeness.
\end{proof}

\noindent\emph{Remark.}
The layered permutations \(\rho_n\) and \(T_n\) also define
a sofic approximation of \(G\).
The additional finite quotient factor makes asymptotic freeness
immediate and avoids a separate fixed-point calculation.

\subsection{An exactly \(H\)-invariant measurable subset}

We now construct a measurable subset \( E_n \subseteq V_n\)
that is invariant under \(\widetilde{\rho}_n(h)\) for every \(h\in H\).
For \(1\leq m\leq n\), define \(b_m:R_n\to\{0,1\}\) by
\begin{equation}\label{eq:digit}
 b_m(r)=
 \begin{cases}
  0,&r\bmod 2^m<2^{m-1},\\
  1,&r\bmod 2^m\geq2^{m-1},
 \end{cases}
\end{equation}
where residues modulo \(2^m\) are represented in
\(\{0,\ldots,2^m-1\}\).
Thus \(b_m\) reads the binary digit in position \(m-1\).
In particular, for \(r,c\in R_n\),
\begin{equation}\label{eq:digit-flip}
 b_m(r+2^{m-1}c)
 \equiv b_m(r)+(c\bmod2)\pmod2.
\end{equation}
Define
\begin{equation}\label{eq:measurable-subset}
 E_n=\bigl\{(i,\ell,x,y)\in V_n:
             i<n,\ b_{n-i}(x)=1\bigr\}.
\end{equation}
The condition \(i<n\) ensures that \(b_{n-i}\) is defined
and that \(E_n\) contains no points of the last layer.

On each layer \(i<n\), exactly half of the values of \(x\)
satisfy the digit condition, independently of \(\ell\) and \(y\).
Therefore,
\begin{equation}\label{eq:subset-measure}
 \mu_n(E_n)=\frac{n}{2(n+1)}\longrightarrow\frac12.
\end{equation}

\begin{lemma}\label{lem:subset-invariant}
For every \(h\in H\), \(\widetilde{\rho}_n(h)E_n=E_n\).
\end{lemma}

\begin{proof}
For \(h=(u,B)\in H\), the permutation \(\widetilde{\rho}_n(h)\)
changes \(x\) on layer \(i<n\) by a multiple of \(2^{n-i}\).
Its residue modulo \(2^{n-i}\), and hence the value of \(b_{n-i}\),
is unchanged.
Membership in \(E_n\) does not depend on \(\ell\) or \(y\),
and \(\widetilde{\rho}_n(h)\) preserves the layer index.
Since the last layer is disjoint from \(E_n\),
this proves \(\widetilde{\rho}_n(h)E_n=E_n\).
\end{proof}

\subsection{Action of the half-translation}

Recall from Lemma~\ref{lem:group-structure} that
\(a=(e_1,I_3)\) lies in \(H\) and that \(w=t^{-1}at\)
corresponds to translation by \(\frac12e_1\).
Set \(W_n=\widetilde{T}_n^{-1}
\widetilde{\rho}_n(a)\widetilde{T}_n\).
Direct evaluation gives
\begin{equation}\label{eq:W}
 W_n(i,\ell,x,y)=
 \begin{cases}
  (i,\ell,x+2^{n-i-1}\ell_1,p_n(w)y),&i<n,\\
  (i,\ell,x,p_n(w)y),&i=n.
 \end{cases}
\end{equation}
Indeed, \(\widetilde{T}_n\) sends layer \(i<n\) to layer \(i+1\),
where \(\widetilde{\rho}_n(a)\) adds \(2^{n-(i+1)}\ell_1\) to \(x\).
It sends layer \(n\) to layer \(0\), where the translation term
vanishes, since \(2^n=0\) in \(R_n\).
On \(Q_n\), the composition acts by left multiplication with \(p_n(w)\).

\begin{proposition}\label{prop:count}
For every \(n\geq1\), we have
\(\mu_n(W_nE_n\triangle E_n)=\frac{n}{2(n+1)}\).
\end{proposition}

\begin{proof}
On layer \(i<n\), put \(m=n-i\).
By \eqref{eq:digit-flip} we have
\[
 b_m(x+2^{m-1}\ell_1)
 \equiv b_m(x)+(\ell_1\bmod2)\pmod2.
\]
Thus membership in \(E_n\) changes exactly when
\(\ell_1\) is odd, independently of \(x\) and \(y\).
Exactly half of the vectors \(\ell\) have this property.
The last layer is preserved by \(W_n\) and is disjoint from \(E_n\),
so it contributes nothing.
By measure preservation,
\begin{align*}
 \mu_n(W_nE_n\triangle E_n)
 &=\mu_n\bigl(\{(i,\ell,x,y)\in V_n: i<n,\ \ell_1\equiv1\pmod2\}\bigr)=\frac{n}{2(n+1)}.\qedhere
\end{align*}
\end{proof}

For a concrete illustration, take \(n=3\), \(i=1\),
and \(\ell=(1,0,0)\), with arbitrary \(y\in Q_3\).
Membership in \(E_3\) is determined by the binary digit
of weight \(2\) in \(x\), with arithmetic modulo \(8\):
\[
 \begin{array}{c|cccccccc}
  x&0&1&2&3&4&5&6&7\\ \hline
  \mathbf1_{E_3}&0&0&1&1&0&0&1&1
 \end{array}
\]
On these points, \(a\) adds \(4\) to \(x\) and preserves membership
in \(E_3\), whereas \(w\) adds \(2\) and flips membership.
The action on \(Q_3\) does not affect membership.
The fraction \(1/2\) comes from counting the vectors \(\ell\)
with odd \(\ell_1\).

\section{Non-stability and consequences}\label{sec:conclusion}

\begin{proof}[Proof of Theorem~\ref{thm:main}]
As observed in Section~\ref{sec:group}, both \(G\) and \(H\) are finitely presented
and residually finite; moreover, \(H\) has property~\(\mathrm{(T)}\).
By Lemma~\ref{lem:group-structure}, \(H\) is infranormal in \(G\), and \(w\) belongs to the normal closure \(N=\langle\!\langle H \rangle\!\rangle\).

By Proposition~\ref{prop:sofic}, the sequence \((\sigma_n)\)
is a sofic approximation of \(G\).
Since \(\sigma_n(s)=\widetilde{\rho}_n(s)\) for every \(s\in S_H\),
Lemma~\ref{lem:subset-invariant} gives
\(
 \mu_n\bigl(\sigma_n(s)E_n\triangle E_n\bigr)=0
\)
for every \(s\in S_H\).
Choose \(v_a\in F(S_H)\) representing \(a\), and represent
\(w\) by \(t^{-1}v_at\). Then
\(\sigma_n(w)=W_n=\widetilde{T}_n^{-1}
\widetilde{\rho}_n(a)\widetilde{T}_n\).
Thus Proposition~\ref{prop:count} gives
\[
 \mu_n\bigl(\sigma_n(w)E_n\triangle E_n\bigr)
 =\frac{n}{2(n+1)}
 \longrightarrow\frac12>0.
\]

Corollary~\ref{cor:criterion} therefore applies with \(c=1/2\).
For every nonprincipal ultrafilter \(\omega\) on \(\mathbb N\),
the induced limit action is not weakly contained in finite actions.
In particular, \(G\) is not stable in finite actions.
\end{proof}

\begin{remark}\label{rem:fixed-algebras-closure}
Let \(\alpha_\omega\) be the limit action on
\((\cB_\omega,\mu_\omega)\) associated with a nonprincipal
ultrafilter \(\omega\) on \(\mathbb N\), and put \(E=[E_n]\).
By Lemma~\ref{lem:subset-invariant} and
Proposition~\ref{prop:count}, the measurable subset \(E\) is
\(H\)-invariant and satisfies
\(
 \mu_\omega(wE\triangle E)=1/2.
\)
Since \(H\leq t^{-1}Ht\leq N\) and \(w\in t^{-1}Ht\),
we obtain
\begin{equation}\label{eq:strict-algebras}
 \cB_\omega^N
 \subseteq \cB_\omega^{t^{-1}Ht}
 =\alpha_\omega(t^{-1})(\cB_\omega^H)
 \subsetneq\cB_\omega^H.
\end{equation}
Here the equality follows by conjugating \(H\), and the
final inclusion is strict because \(E\) is fixed by \(H\)
but not by \(w\).
Thus \(\cB_\omega^H\) is not preserved by \(G\).

For every measured \(G\)-Boolean algebra weakly contained
in finite actions, Corollary~\ref{cor:weak-inequality}
gives \(\cB^H=\cB^N\).
Since \(N\) is normal, \(G\) preserves this algebra,
and its action on it factors through \(G/N\).
Since \(G\) surjects onto \(\Z\), it does not have
property~\((T)\). Thus Kun and Thom's Theorem~C
\cite{KT} does not apply to this example.

Finally, the closure of \(H\) in the profinite topology
of \(G\) is exactly \(N\).
By Lemma~\ref{lem:finite-quotient}, \(H\) and \(N\) have
the same image in every finite quotient of \(G\), so
\(N\subseteq\overline H^{\,\mathrm{pf}}\).
Conversely, if \(g\notin N\), then
Lemma~\ref{lem:group-structure} gives \(\chi(g)=k\neq0\).
For any integer \(m>|k|\), the finite quotient
\(\chi\bmod m:G\to\Z/m\Z\) separates \(g\) from \(H\).
Thus \(\overline H^{\,\mathrm{pf}}=N\).
Since \(H\subsetneq N\), the subgroup \(H\) is not
separable, although \(G\) is residually finite.
\end{remark}

\section*{Acknowledgements}
The author thanks Vadim Alekseev and Andreas Thom for valuable discussions
that helped simplify the proof and improve the exposition.

\begingroup
\raggedright
\bibliographystyle{amsalpha}
\bibliography{references}
\endgroup

\end{document}